\documentclass[a4paper,abstracton,11pt]{scrartcl}
\usepackage[utf8]{inputenc}
\usepackage{amsfonts,amsthm,amssymb,amsmath}
\usepackage[inline]{enumitem}
\usepackage{color}
\usepackage{graphicx}
\usepackage{authblk}

\usepackage{tikz}

\newtheorem{ctr}{}[section]
\newtheorem{theorem}[ctr]{Theorem}
\newtheorem{lemma}[ctr]{Lemma}
\newtheorem{claim}[ctr]{Claim}

\newtheorem{conjecture}[ctr]{Conjecture}

\newcommand{\Aut}{\operatorname{Aut}}
\newcommand{\Orb}{\textnormal{Orb}}
\renewcommand{\leq}{\leqslant}
\renewcommand{\geq}{\geqslant}

\newcommand{\hg}{H_g}

\title{Distinguishing infinite star-free graphs}

\author{Marcin Stawiski \\ stawiski@agh.edu.pl}

\affil{AGH University of Science and Technology,\\ Faculty of Applied Mathematics, \protect\\al. Mickiewicza 30, 30-059 Krakow, Poland}

\begin{document}

\maketitle

\begin{abstract}
Call a colouring of a graph \emph{distinguishing} if the only automorphism of this graph which preserves said colouring is the identity.  Let $H$ be an arbitrary graph. We say that a graph $G$ is \emph{$H$-free} if $G$ does not contain an induced subgraph isomorphic to $H$. Kargul, Musiał, Pal and Gorzkowska showed that if $n$ is a natural number greater than two, then every finite connected $K_{1,n}$-free graph of order at least six  admits a distinguishing edge colouring with at most $n-1$ colours. We extend this result to all locally finite connected $K_{1,n}$-free graphs of order at least six.


\end{abstract}

\section{Introduction}

 We say that a graph is \emph{locally finite} if its every vertex has finite degree.  Let $H$ be an arbitrary graph. We say that a graph $G$ is $H$\emph{-free} if $G$ does not contain an induced subgraph isomorphic to $H$. 
Call a (not necessarily proper) colouring of a graph \emph{distinguishing} if the only automorphism which preserves it is the identity.
 The \emph{distinguishing index} $D'(G)$ of a graph $G$ is the least number of colours in a distinguishing edge colouring of $G$. The least number of colours in a distinguishing vertex colouring of a graph $G$ is called its \emph{distinguishing number}, and it is denoted by $D(G)$. 
Distinguishing colourings   were firstly studied by Babai  \cite{BAB} in 1977 in the case of  infinite regular trees under the name \emph{asymmetric colourings}. They arose during his study of the graph isomorphism problem, which eventually led to his proof of the existence of quasi-polynomial algorithm for this problem (see \cite{babaiisomorphism}). Distinguishing colourings in a more general group-theoretic context were introduced by Cameron, Neumann and Saxl \cite{primitive} in 1984. Finally, distinguishing edge colourings were firstly studied by Kalinowski and Pilśniak \cite{KP} in 2015. Note that the problem whether a given graph has the distinguishing index at most two is closely related (but not equivalent) to the existence of its asymmetric spanning subgraph.

Denote by $\Delta(G)$ the supremum of degrees of vertices of a graph $G$. For a connected finite graph $G$, we have $D(G)\leq \Delta(G)$ unless $G$ is isomorphic to $C_5$, $K_{\Delta+1}$, or  $K_{\Delta,\Delta}$ (see \cite{CT,klavzar}). Similarly, it has been proved in \cite{KP} that if $G$ is a connected finite graph, then $D'(G) \leq \Delta(G)$ unless $G\cong K_2$, $G\cong C_3$, $G\cong C_4$, or $G\cong C_5$.
 For a connected infinite graph $G$ the optimal upper bounds for its distinguishing number and its distinguishing index are both equal to $\Delta-1$ unless $G$ is a double ray \cite{LPS,PS}.
 Moreover, Lehner and Smith have proved that the distinguishing index of a connected graph cannot be much greater than its distinguishing number  (see  \cite{lehner:smith}). However, in many cases the distinguishing index of a graph is significantly smaller than its distinguishing number. Broadly studied classes of graphs for which we can significantly improve the upper bounds for the distinguishing index but not for the distinguishing number include traceable finite graphs \cite{Pilsniakimproving}, regular connected graphs of finite degree \cite{regular}, and finite connected star-free graphs \cite{gorzkow}.

In this paper we investigate distinguishing edge colourings of connected  infinite locally finite star-free  graphs.  Note that there exist connected  infinite locally finite  claw-free graphs  with the distinguishing number equal to $\aleph_0$. Hence, the problem we investigate in this paper is trivial for vertex colourings but may still be interesting for edge colourings.  As mentioned before, for finite graphs this problem has  already been investigated in \cite{gorzkow}, where the following result has been obtained.

\begin{theorem}[Gorzkowska, Kargul, Musiał, Pal 2020 \cite{gorzkow}]\label{thm:gorzkow}
Let $n$ be an arbitrary natural number greater than two. If $G$ is a finite connected $K_{1,n}$-free graph of order at least six, then $D'(G)\leq n-1$.
\end{theorem}

The main result of this paper is the extension of the  theorem above to all locally finite graphs of order at least six, which is obtained by using different techniques that in the proof of the said theorem.

\begin{theorem}\label{thm:bezgwiazd}
Let $n$ be an arbitrary natural number greater than two. If $G$ is a locally finite connected $K_{1,n}$-free graph of order at least six, then $D'(G)\leq n-1$.
\end{theorem}

To see that the  bound above is tight for finite graphs, it is enough to consider  $K_{1,n-1}$. Note that there exist  infinite connected $K_{1,3}$-free graphs  which have non-trivial automorphisms. Hence, this bound is also tight in the case of infinite locally finite  $K_{1,3}$-free graphs. 
The question whether the bound in Theorem \ref{thm:bezgwiazd} is  tight for any $n\geq4$ in the case of infinite graphs remains open.

\begin{conjecture}
For every natural number $n$  greater than three there exists a connected infinite locally finite graph $G$ such that $D'(G) = n-1$.
\end{conjecture}

Note that every graph $G$ with the maximum degree $\Delta$ is  $K_{1,\Delta+1}$-free. If for every $n$ greater than three there are no graphs satisfying the statement of the  conjecture above, then it would imply the mentioned bound $D'(G)\leq \Delta(G)-1$ for every connected infinite locally finite graph $G$ of maximum degree at least four.

\section{Definitions and auxiliary results}

Let $\Omega$ be an arbitrary set, let $\Gamma$ be a group acting on $\Omega$, and let $X,Y\subseteq \Omega$. The \emph{orbit} $\Orb_\Gamma(X)$ of $X$ with respect to $\Gamma$ is the set $\{\varphi(X)\colon  \varphi \in \Gamma \}$. If $X=\{ x\}$, then we refer to the orbit of $x$ instead of $\{x\}$.
We say that $X$ is \emph{fixed} if every $\varphi \in \Gamma$ acts trivially on $X$. If $X=\{ x\}$, then we say that $x$ is fixed instead of $\{x\}$. We say that $X$ is \emph{stabilized} if $\varphi(X)=X$ for every $\varphi \in \Gamma$. 

Let $c$ be a partial colouring of $\Omega$, and let $\varphi\in \Gamma$. 
We say that $c$ \emph{breaks} $\varphi$, if there exists  $x\in \Omega$ such that both $c(x)$ and $c(\phi(x))$ are defined and $c(x)\neq c(\varphi(x))$. Otherwise, we say that $\varphi$ \emph{preserves} $c$. Call a partial colouring $c$ of $\Omega$ \emph{distinguishing} if the only element of $\Gamma$ which preserves  $c$ is the identity. 

Let $R$ be a subset of the set of vertices of graph $G$. A pair $(G,R)$ is called a \emph{rooted graph}, and we say that $R$ is the \emph{root} of $(G,R)$. If $R=\{r\}$, then we refer to $r$ as the root of $(G,R)$. 
If $R=V(H)$ is the set of vertices of some subgraph $H$ of $G$, then we say that $H$ is the \emph{root} of $(G,H)$. Automorphisms of a rooted graph $(G,R)$ are exactly these automorphisms of $G$ which fix $R$.
The set of automorphisms of $(G,R)$ will be denoted by $\Aut(G,R)$, or by $\Aut(G)$ if $R=\emptyset$.
In this paper $\Omega$  usually denotes the set of edges of some graph. The meaning of  $\Gamma$ (if not stated explicitly) shall follow from the context.
For graph notions which are not defined in this paper  see \cite{Diestel}.

We say that colourings $c,d$ of $\Omega$ are \emph{non-isomorphic} if for each $\varphi \in \Gamma$ there exists $x \in \Omega$ such that $c(x)\neq d(x)$. For the proof of the main theorem of this paper we need to obtain non-isomorphic colourings of some finite rooted graphs.
The \emph{hourglass} $\hg$ is the graph obtained by gluing two copies of $K_3$ at a common vertex. The vertex of degree four in $\hg$ is called its \emph{central} vertex.

\begin{theorem}\label{thm:smallgraphs}
Let $n$ be an arbitrary natural number greater than two. If $r$ is a vertex of a connected finite $K_{1,n}$-free graph $G$, then one of the following cases holds:
\begin{enumerate}[label=\textnormal{(C\arabic*)}]
    \item there exist at least $n-1$ non-isomorphic distinguishing edge colourings of the rooted graph $(G,r)$ using at most $n-1$ colours, \label{itm:thm:smallgraphs1}
    \item $n=3$, $G \cong \hg$, and $r$ is the central vertex of $G$,\label{itm:thm:smallgraphs2}
    \item $G \cong K_{1,n-1}$, and $r$ is the central vertex of $G$, or \label{itm:thm:smallgraphs3}
    \item $G \cong K_1$.\label{itm:thm:smallgraphs4}
\end{enumerate}
\end{theorem}
\begin{proof}
Assume that none of the cases \ref{itm:thm:smallgraphs2}--\ref{itm:thm:smallgraphs4} holds. We shall prove that the condition \ref{itm:thm:smallgraphs1} is satisfied for $(G,r)$.
By Theorem \ref{thm:gorzkow}  every connected finite $K_{1,n}$-free graph which does not have distinguishing index at most $n-1$ has order at most five.
These are exactly the graphs isomorphic to $K_2$, $K_3$, $K_4$, $K_5$, $C_4$, or $C_5$. Let $G$ be one of these graphs, and let $r$ be an arbitrary vertex of $G$. Notice that $G$ contains a Hamiltonian path $P$ with endvertex $r$. 
Hence, the colouring which assigns colour blue to every edge of $P$, and which assigns colour red to all  the remaining edges of $G$ is a distinguishing colouring of $(G,r)$ using at most two colours. 
We can obtain $n-1$ non-isomorphic colouring if we swap colour blue with another colour.
Additionally, it follows from Theorem \ref{thm:gorzkow} that for every connected finite $K_{1,n}$-free graph $G$ and every $r \in V(G)$ there exists a distinguishing colouring of $(G,r)$ using at most $n-1$ colours.

Now consider the induction on the size of $G$. By the previous paragraph the statement of the theorem holds for $K_2$. Assume that the statement of the theorem holds for every connected finite $K_{1,n}$-free graph of size less than $||G||$.

First, suppose that  $G[N(r) \cup \{r\}]=G$.
Take any $x \in N(r)$, and consider its orbit $\Orb_\Gamma(x)$ with respect to the group $\Gamma=\Aut(G,r)$.
If $\Orb_\Gamma(x) \neq N(r)$, then define $G_1=G[\Orb_\Gamma(x) \cup\{r\}]$ and $G_2=G[(N(r)\cup \{r\}) \setminus \Orb_\Gamma(x)]$. By the inductive hypothesis, for every $i \in \{1,2\}$ there exist at least $n-1$ non-isomorphic distinguishing colourings  of $(G_i,r)$ if neither $G_i \cong K_{1,n-1}$,  $G_i \cong K_1$,
nor $G_i \cong \hg$ (if $n=3$). If for any $i \in \{1,2\}$ there exist $n-1$ non-isomorphic distinguishing colourings of $(G_i,r)$, then we colour $(G_1,r)$ and $(G_2,r)$ with  distinguishing colourings, and we colour the remaining edges of the graph arbitrarily. We can obtain at least $n-1$ non-isomorphic distinguishing colourings of $G$ by choosing one of the non-isomorphic colourings for $(G_1,r)$ or $(G_2,r)$.

If neither $(G_1,r)$ nor $(G_2,r)$ has $n-1$ non-isomorphic distinguishing colourings, then there exists an edge joining the vertices of $G_1-r$ and $G_2-r$. Otherwise, $G$ would contain an induced $K_{1,n}$. We colour $(G_1,r)$ and $(G_2,r)$ with  distinguishing colourings, and we colour all the remaining edges of $G$ with one colour. By choosing this colour, we can obtain $n-1$ non-isomorphic distinguishing colourings of $G$.

If $\Orb_\Gamma(x)=N(r)$, then it contains at most $k=n-1$ isomorphic components, say $X_0, \dots, X_{k-1}$. 
For every $i \in \{0,\dots,k-1\}$ we choose one edge $rx_i$ such that $x_i \in X_i$. Now let $k \leq n-2$. We choose $j \in \{ 0, \dots, n-2\}$, and for every $i \in \{0,\dots,k-1\}$ we colour the edge $rx_i$ with  colour $i$, each of the remaining edges joining $r$ and $X_i$ with colour $n-1$, and we colour each $(X_i,x_i)$ with a distinguishing edge colouring. Notice that this colouring is distinguishing for $(G,r)$.
We can obtain $n-1$ non-isomorphic distinguishing colourings if we choose $j \in \{ 0, \dots, n-2\}$ and  for every $i \in \{1,\dots,k-1\}$ we recolour each edge of colour $i$ with colour $(i+j) \mod (n-1)$.

If $\Orb_\Gamma(x)=N(r)$ and $k=n-1$, then each $X_i$ is a complete graph. Otherwise, $G$ would contain an induced $K_{1,n}$. 
If $X_i \cong K_1$, then $G \cong K_{1,n-1}$. If $X_i \ncong K_1$, then we colour one edge $rx_0$ joining $X_0$ and $r$ with colour $0$ and the rest of them with colour $1$, for every $i \in \{1, \dots, k-1\}$ we colour one edge $rx_i$ joining $X_i$ and $r$ with colour $0$ and the  remaining edges joining $r$ and $X_i$ with colour $i$. We colour $(X_i,x_i)$ with a distinguishing colouring (isomorphic for each $i$), and we colour $(X_0,x_0)$ with a non-isomorphic distinguishing colouring. If $k=2$ and $X_i \cong K_2$, then $n=3$ and $G \cong \hg$. Otherwise, we can obtain $n-1$ non-isomorphic colourings by choosing $j \in \{0, n-1 \}$ and increasing the colour of each edge in $X_0+r$ by $j \mod (n-1)$.

Suppose that $H=G[N(r) \cup \{r\}] \neq G$.
Let $c$ be a distinguishing edge colouring of $G$ using colours from $\{1,\dots n-1\}$.
We can obtain $n-1$ non-isomorphic distinguishing edge colourings from $c$ by recolouring $H=G[N(r) \cup \{r\}]$ with one of $n-1$ non-isomorphic colourings, which exist by the inductive hypothesis if $(H,r)$ satisfies neither \ref{itm:thm:smallgraphs2} nor \ref{itm:thm:smallgraphs3}. Therefore, we can assume that either \ref{itm:thm:smallgraphs2} or \ref{itm:thm:smallgraphs3} holds for $(H,r)$.

It follows that $H \neq G$. Hence, $N(x) \setminus (N(r) \cup \{r\})$ is non-empty for some $x \in N(r)$. Let $H'=G[N(x) \setminus (N(r) \cup \{r\})]$. Notice that $H'$ is not isomorphic to $K_{1,n-1}$ nor to the hourglass (if additionally $n=3$). Otherwise, $G[N(x) \setminus (N(r) \cup \{r\})]$ would contain $K_{1,n}$. We can obtain $n-1$ non-isomorphic distinguishing edge colourings of $(G,r)$ from $c$ by recolouring $(H',x)$ with one of $n-1$ non-isomorphic distinguishing edge colourings of $(H',x)$, which exist by the inductive hypothesis.
\end{proof}
Note that $K_{1,n}$ is a connected finite graph without $K_{1,n}$ which has exactly $n-1$ non-isomorphic colouring.
Assume that $G$ is a connected finite $K_{1,n}$-free graph, $H$ is an induced proper subgraph of $G$ isomorphic to $K_{1,n-1}$ or  to $\hg$ (if $n=3$). If a vertex $v$ is a neighbour in $G$ of some vertex of $H$, then $v$ is a neighbour of $x \in V(H)$ which is not the central vertex of $H$. Hence, by Theorem \ref{thm:smallgraphs}, there exist at least $n-1$ non-isomorphic distinguishing edge colourings of $(H,x)$. This observation shall be useful in the proof of Theorem \ref{thm:bezgwiazd}.

To extend Theorem \ref{thm:gorzkow} to all connected  locally finite graphs we also need the following  version of Lemma 4 from \cite{PS}. The proofs of both lemmas are similar to each other but instead of arbitrary rays in the proof of Lemma \ref{lem:rodzinaind} we consider induced rays. We leave the details of the proof to the reader.

\begin{lemma}\label{lem:rodzinaind}
If $G$ is an infinite locally finite connected graph, then there exists a maximal subgraph of $G$ such that its every component is an induced ray. \qed
\end{lemma}

\section{Main result} \label{sec:nostars}

Now we can proceed to the proof of  Theorem \ref{thm:bezgwiazd}.

\begin{proof}
The case when $G$ is a finite graph follows from Theorem \ref{thm:gorzkow}. Hence, we can assume that $G$ is infinite.
Let $H' \neq K_1$ be a connected finite  $K_{1,n}$-free graph, and let $H\subseteq G$ be isomorphic to $H'$. We write $w \sim v$ if  $v, w \in V(H)$ and there exists an automorphism of $H$ which maps $w$ into $v$. Take any $w\in V(H)$. Assume that if $H$ is isomorphic to $K_{1,n-1}$ or  $H$ is isomorphic to the hourglass (if additionally $n=3$), then $w$ is not the central vertex of $H$. By Theorem \ref{thm:smallgraphs}, for each such $w$ we can find non-isomorphic distinguishing colourings $c(H,w,0)$ and $c(H,w,1)$ of $(H,w)$ such that they do not depend on the choice of  graph $H$ isomorphic to $H'$ nor of the choice of $w\in [v]_\sim$.

Let $F$ be a subgraph of $G$ from  Lemma \ref{lem:rodzinaind}. Denote by ${\mathcal{R}}= \{R_i \colon  2\leqslant i<\alpha \}$ the family of components of  $F$  for some $\alpha \in \{3,\dots,\omega \}$. Further, denote $R_i=v_{0,i}v_{1,i}\dots$ for every natural $i$ satisfying $2\leqslant i<\alpha$. Note that every element of $\mathcal{R}$ is an induced ray and $\alpha=\omega$ if and only if the family $\mathcal{R}$ is infinite. Denote by $\Gamma$ the group of automorphisms of $(G,F)$.

For every $u \in R \in \mathcal{R}$ denote by $u^+$ its successor in $R$ and by $u^-$ its predecessor, which exists if and only if $u$ is not the endvertex of $R$. Let $v\in R\in \mathcal{R}$, and let $w\in V(G) \setminus V(F)$ be a neighbour of $v$. If $v^+w$ is an edge in $G$, then we say that the edge $vw$ is \emph{supported} and  that $v^+w$ is the \emph{supporting edge} of $vw$. For $u \in V(F)$ denote by $\mathcal{B}(u)$ the set of these components of $G-V(F)$ which are connected to $u$ in $G$. We  now prove the following claim about supported and supporting edges.

\begin{claim}\label{stw:nostars}
If $v\in R \in \mathcal{R}$ and  $B\in \mathcal{B}(v)$, then $|\Orb_\Gamma(B)| \leqslant n-2$. Furthermore, if additionally $|\Orb_\Gamma(B)|=n-2$, then $v$ is the endvertex of $R$ or there exists a vertex $b \in B \in \Orb_\Gamma(B)$ such that $vb$ is a supported edge or a supporting edge.
\end{claim}
\begin{proof} Consider an arbitrary induced star which contains edges joining $v$ and every component from $\mathcal{B}(v)$. As $G$ does not contain $K_{1,n}$ as an induced subgraph, we obtain $|\mathcal{B}(v)|\leq n-1$.

Suppose that there exists a component $B\in \mathcal{B}(v)$ such that its orbit $\Orb_\Gamma(B)$ has $n-1$ elements. The vertex $v^+$ is connected to some element of $\Orb_\Gamma(B)$. Otherwise, there exists an induced $K_{1,n}$ in $G$. Vertex $v^+$ is  connected to some element of  $\Orb_\Gamma(B)$. Hence, $v^+$ is connected to every element of  $\Orb_\Gamma(B)$. 
It follows by induction that every vertex of $R$ is connected to every element of $\Orb_\Gamma(B)$, contradicting the local finiteness of $G$.

Suppose that $|\Orb_\Gamma(B)|=n-2$. If none of the cases  from the statement of the theorem holds, then there exists an induced star $K_{1,n}$ which contains edges $vv^+$, $vv^-$, and an edge to each element of $\Orb_\Gamma (B)$. This contradicts the fact that $G$ does not contain an induced $K_{1,n}$, and it finishes the proof of the claim.
\end{proof}

Now we define the partition of $\mathcal{R}$ into sets $\mathcal{R}_0$ and $\mathcal{R}_\infty= \mathcal{R} \setminus \mathcal{R}_0$. Let $\mathcal{R}_0$ be the family of these rays $R \in \mathcal{R}$ which have a subray whose vertices do not have a neighbour outside $F$. Note that $R \in \mathcal{R}_\infty$ if and only if $R \in \mathcal{R}$  and there are infinitely many components of $G- V(F)$ which are connected to $R$ in graph $G$.
Define $I_0=\{i<\alpha\colon  R_i\in R_0 \}$ and $I_\infty=\{i<\alpha\colon  R_i\in R_\infty \}$.

Let $( v'_{j,k}\colon  j<\omega)$ be the enumeration of these vertices $v$ of ray $R_k\in \mathcal{R}_\infty$ for which $\mathcal{B}(v)\neq \emptyset$, ordered according to the order of vertices in $R_k$.
We shall  construct set $M=\{m_k\colon  k\in I_\infty  \}$ whose elements are natural numbers greater than one. If $m_k \in M$, then we say that the vertex $v'_{m_k,k}$ is the \emph{favourite vertex} of the ray $R_k$. Favourite vertices shall be used to distinguish rays in $R_\infty$ from each other. We construct  $M$ in such a way that the following conditions shall be satisfied  for all favourite vertices $u$, $v$:
\begin{enumerate}[labelindent=\parindent,leftmargin=*, label=(M\arabic*)]
    \item if $v \neq u$, then  $\mathcal{B}(v)\cap\mathcal{B}(u)=\emptyset$,\label{enum:M1}
    \item if $r$ is the endvertex of a ray $R$ and $v$ is the favourite vertex of  $R$, we have $\mathcal{B}(v)\cap \mathcal{B}(r)=\emptyset$,\label{enum:M2}
    \item if $B\in \mathcal{B}(v)$, then $\Orb_\Gamma (B)$ has at most $n-3$ elements or there exists a supported edge incident to $v$.\label{enum:M3}
    \end{enumerate}
We construct the set $M$ by induction on $l$. Assume that we have already constructed the set $M_l=\{m_k\colon  k< l, k \in I_\infty \}$ and that $M_l \neq M$. If $l \notin I_\infty$, then we proceed to the next index.
Therefore, we can assume that $l \in I_\infty$. Let $j'$ be the least natural number greater than every element of $M_l \cup \{1 \}$ such that for every $j$ satisfying $j\geq j'$, we have $\mathcal{B}(v'_{j,l})\cap \mathcal{B}(v_{0,l})=\emptyset$ and for every $k\in I_\infty, k < l$, we have $\mathcal{B}(v'_{j,l})\cap \mathcal{B}(v'_{m_k,k})=\emptyset$.
Such number $j'$ exists because  $\mathcal{B}(v)$ is finite for every $v\in V(F)$ and for every $k \in I_\infty$ the set $\{ v'_{j,k}\colon  j<\omega\}$ is infinite. If $|\mathcal{B}(v'_{j',l})|\leq n-3$ holds or $v'_{j',l}$ has a supported edge, then we put $m_l=j'$. Otherwise, $v'_{j'+1,l}$ does not have a supporting edge. Therefore, by Claim \ref{stw:nostars} 
the vertex $v'_{j'+1,l}$ has a supported edge or $|\mathcal{B}(v'_{j'+1,l})|\leq n-3$. In both of these cases we put $m_l=j'+1$. It follows from the construction that the conditions \ref{enum:M1}--\ref{enum:M3} are satisfied for all favourite vertices $u,v$.

Now we can proceed with the construction of a distinguishing edge colouring of $G$ using colours from the set $C=\{ 1,\dots,n-1\}$. We shall refer to   colour 1 as \emph{blue}, colour 2 as \emph{red} and  colour 3 as \emph{yellow}. Note that yellow is an element of $C$ if and only if $n>3$. We colour every edge joining distinct rays from $\mathcal{R}$ with red. Moreover, if $R  \in \mathcal{R}_\infty$, then we colour every edge of $R$ with blue. By induction on $i\in I_0$  we choose a natural number $k(i)$   in such a way that the following conditions are satisfied:

\begin{enumerate}[labelindent=\parindent,leftmargin=*, align=left,label=(K\arabic*)]
\item \label{k-a} $k(i)$ is  even number greater than two,
    \item \label{k-b} the subray of $R_i$ with endvertex $v_{\frac12 k(i),i}$ does not have any neighbour outside $F$,

    \item \label{k-c} for $i\neq j$ we have $k(i) \neq k(j)$,
    
    \item  \label{k-d} for $i \neq j$ vertices $v_{k(i),i}$ and $v_{2k(i)+1,i}$ are not neighbours of $v_{k(j),j}^-$ nor $v_{2k(j),j}$.
\end{enumerate}
Assume that we have already chosen the set $\{k(i)\colon  i<l,i \in \mathcal{R}_0 \}$.
There exists even number $m'$ such that for every even number $m$ satisfying $m \geq m'$  if we put $k(l)=m$, then the conditions \ref{k-a}--\ref{k-d} hold for each pair $\{i,j\}$ satisfying $i,j \leq l$. 
The existence of $m'$ follows from the fact that every ray from $\mathcal{R}_0$ has an infinite subray which does not have any neighbour outside $F$ and the conditions \ref{k-c} and \ref{k-d} for $i,j\leq l$ are satisfied for all but finitely many numbers $m=k(l)$. We put $k(l)=m'$, and we repeat the procedure for the next index $l$.

For $i\in I_0$ define $P(i)=v_{0,i}, \dots,v_{k(i),i}^-$, $P'(i)= v_{k(i),i}\dots,v_{2k(i),i}$, and further denote by $S_{R_i}$ the subray of $R_i$ with endvertex $v_{2k(i)+1,i}$.
For every $i\in I_0$ we colour the edges of $P(i),P'(i)$, and $S_{R_i}$ with blue, and we colour the remaining  two edges of $R_i$ with red. Now we colour the edges joining the vertices of $F$ and the vertices from $G-V(F)$ by induction with respect to $j$ and $i$.

$(*)$ Let $R_i \in \mathcal{R}$ be the ray with the least index among these for which there exists an uncoloured edge incident to some vertex of $R_i$. Similarly, let $v_{j,i}\in R_i$ be the vertex with the least index $j$ among these vertices of $R_i$ which have an uncoloured incident edge. For clarity denote $v=v_{j,i}$.
In the induction step we colour every edge incident to $v$  which was not coloured earlier. 
We choose one edge joining $v$ and each component of $\mathcal{B}(v)$.
If a component $B$ of $\mathcal{B}(v)$ is isomorphic to $K_{1,n-1}$ or $\hg$ (and additionally $n=3$), then $v$ is connected to a vertex of $B$ which is not its central vertex.
In said case we always choose an edge connected to vertex of $B$ which is not its central vertex. If we can, we choose a supported edge, next we choose a supporting edge, and then one of the remaining edges.

We colour the chosen edges with  some additional conditions. Denote by $B_w\in \mathcal{B}(v)$ the component of $G-V(F)$ containing $w$.
If any of the conditions below is satisfied, and we are supposed to colour $vw$ with a colour different from red, then we instead  colour it with red. 
\begin{enumerate}[labelindent=\parindent,leftmargin=*, align=left,label=(R\arabic*), series=Rnostars]
     \item \label{item:red1} $B_w$ is connected to the favourite  vertex $u\neq v$ of some $R \in \mathcal{R}$.
     \item \label{item:red2} There exists an edge joining $B_w$ and some vertex $u \in V(F)$ which is already coloured with a colour different from red.
     \end{enumerate}
Furthermore, if we are supposed to colour $vw$ with blue and the following condition is satisfied, then we instead colour it  with red.
 \begin{enumerate}[resume*=Rnostars]
     \item \label{item:red3} $B_w$ is isomorphic to $K_1$, and  $v$ is neither the favourite vertex of $R_i$ nor the endvertex of $R_i$.
 \end{enumerate}
 Moreover, if the  condition above is satisfied,  $|\Orb_\Gamma(B)|\leq n-3$, and we are supposed to colour $vw$ with yellow, then we instead colour $vw$ with red.
 
$(**)$ We now proceed to the colouring of the chosen edges. We consider the orbits of elements of $\mathcal{B}(v)$ with respect to $\Gamma=\Aut(G,F)$. If $|\Orb_\Gamma(B)|\leqslant n-3$, then yellow is in $C$. We colour one of the chosen edges joining $v$ and some component of $\Orb_\Gamma(B)$ with yellow, and we colour all the remaining chosen edges joining $v$ and the components in $\Orb_\Gamma(B)$ with distinct colours except yellow, blue and red.

Suppose $|\Orb_\Gamma(B)|= n-2$. By Lemma \ref{stw:nostars}, $v$ is  the endvertex of $R_i$, $v$ has a supported edge to some component in $\Orb_\Gamma(B)$, or $v$ has a supporting edge to some component in $\Orb_\Gamma(B)$. Now we consider cases.
\begin{enumerate}[labelindent=\parindent,leftmargin=*, align=left,label=(C\arabic*), series=Cnostars]
\item If $v$ is the endvertex of $R_i$ and $B$ is isomorphic to $K_1$ or $v$ is the endvertex of $R_i$ and $n=3$, then we colour one of the chosen edges to $\Orb_\Gamma(B)$ with blue, and we colour the remaining chosen edges to $\Orb_\Gamma(B)$ with distinct colours except blue. 

\item If $v$ is the endvertex of $R_i$, $B$ is not isomorphic to $K_1$, and $n>3$, then we colour two of the chosen edges to $\Orb_\Gamma(B)$ with blue, and we colour the remaining chosen edges with distinct colours except blue and red.

\item If $v$ is not the endvertex of $R_i$ and we have chosen a supported edge to $\Orb_\Gamma(B)$, then we colour one of such edges with blue, and we colour the remaining chosen edges with distinct colours except blue and red. 

\item If $v$ is not the endvertex of $R_i$ and $v$ does not have a supported edge to $\Orb_\Gamma(B)$, then  $v$ has a supporting edge to $\Orb_\Gamma(B)$. In that case, at least one of the supporting edges, say $e$, to $\Orb_\Gamma(B)$ was chosen.
Notice that the edge supported by $e$ was coloured when we were considering $v^-$. By the conditions \ref{item:red1}--\ref{item:red3} and the fact that we have already coloured the edge supported by $e$, we have to colour $e$ with red. We colour the remaining chosen edges joining $v$ and $\Orb_\Gamma(B)$ with distinct colours except blue and red.
\end{enumerate}
Next, we colour all the edges joining $v$ and $\Orb_\Gamma(B)$ which are not chosen, with red. If there exists an orbit $\Orb_\Gamma(B')$ of some $B' \in \mathcal{B}(v)$ which was not considered yet, then we put $B=B'$ and we return to $(**)$.
If there exists an uncoloured edge incident to some vertex of $R_i$, then we increase $j$ and we return to $(*)$. Similarly, if there is no such edge but there exists an uncoloured edge incident to some vertex in $V(F)$, then  we increase $i$,  we put $j=0$, and we return to $(*)$.

After the induction with respect to $j$ and $i$ all the edges incident to vertices of $V(F)$ are coloured. Notice that from the construction of the partial colouring, we obtained that if a component $B$ of $G-V(F)$ is not isomorphic to $K_1$, then there exists exactly one edge joining $B$ and $F$ of colour different from red. This follows from the conditions \ref{item:red1}--\ref{item:red3} and the fact that if none of said conditions holds, then we do not colour the chosen edge with red unless it is a supporting edge.

Let $\Gamma'$ be a set of these automorphisms of $(G,F)$ which preserve partial colouring we have defined so far, and let $B$ be some component of $G-V(F)$.
If $B$ is isomorphic to $K_1$, then denote by $w_B$ the only vertex of $B$. Otherwise, denote by $w_B$ the unique vertex of $B$ which has an edge to $F$ of colour different from red. Notice that from the construction of the partial colouring, it follows that $B$ is stabilized with respect to $\Gamma'$ unless $B'$ is incident to the endvertex $r$ of some ray in $\mathcal{R}$.
Moreover, if $B$ is not stabilized with respect to $\Gamma'$, then $B$ is not isomorphic to $K_1$ and there exists $B' \in \Orb_{\Gamma'}(B)$ distinct from $B$ such that both $B$ and $B'$ are connected with blue edges to  the endvertex $r$ of some $R\in \mathcal{R}$. We colour the rooted graph $(B,w_B)$ with  colouring $c(B,w_B,0)$, and we colour the rooted graph $(B',w_{B'})$ with  colouring $c(B',w_{B'},1)$. Components $B$ and $B'$ are stabilized by the group $\Gamma''\subseteq \Gamma'$ of these automorphisms of $(G,F)$ which preserve the so far defined  partial colouring. We repeat the procedure for each component $B$ which is not coloured and is not stabilized with respect to $\Gamma'$.

Next, if $B \not\simeq K_1$ and $B$ is not connected to any of the favourite vertices, we colour $(B, w_B)$ with  colouring $c(B,w_B,0)$. If $B \not\simeq K_1$ and $B$ is connected to one of the favourite vertices, then $w_B$ is connected to this vertex with a non-red edge and we colour $(B, w_B)$ with  colouring $c(B,w_B,1)$. After repeating this procedure for every component $B$ of $G-V(F)$ we obtain a distinguishing colouring $c$ of graph $(G,F)$.

To show that the colouring $c$ is distinguishing for $G$ it remains to prove that $F$ is fixed by $c$ with respect to the group of automorphism of $G$. First, we show that the family $\mathcal{R}_1=\mathcal{R}_\infty\cup \{S_R\colon  R\in \mathcal{R}_0 \}$ is stabilized. 
 Assume that there exists $Q\notin \mathcal{R}_1$ which is a maximal blue induced ray. Every component of $G-V(F)$ is connected to vertices of $F$ with at most one blue edge. Hence, $Q$ has to contain a subray of some ray from  $\mathcal{R}_1$.
Notice that if $u\in R\in \mathcal{R}$ is not the endvertex of $R$ and $u$ is connected to some vertex $w \in V(G) \setminus V(F)$ with a blue edge, then $uw$ is a supported edge and  $u^+$ is connected to $w$. Thus, a blue ray which contains vertices $u, u^+$, and $w$ cannot be induced.
If $u\in R\in \mathcal{R}$ is the endvertex of $R$ and $u$ is connected to vertex $w \in V(G) \setminus V(F)$ with a blue edge, then $G[V(R) \cup \{w\}]$ is not an induced ray because $R\subsetneq G[V(R) \cup \{w\}]$ and $R$ is the maximal induced blue ray. The latter follows from properties of family $F$ from Lemma \ref{lem:rodzinaind} and the fact that the edges joining distinct rays from $\mathcal{R}$ are red. Therefore, we proved that the family $\mathcal{R}_1=\mathcal{R}_\infty\cup \{S_R\colon  R\in \mathcal{R}_0 \}$ is stabilized by $c$.

Every edge joining $\{S_R\colon  R\in \mathcal{R}\}$  and the rest of the graph is red, and every ray from the family $\mathcal{R}_\infty$ has at least one incident edge of colour different from red (in particular the edge incident to its favourite edge). It follows that the sets $\{S_R\colon R\in \mathcal{R}_0\}$ and $\mathcal{R}_\infty$ are stabilized by $c$.

Now we show that the rays from the family $\mathcal{R}_0$ are fixed. Let $i \in I_0$. Notice that $P'(i)$ is the only maximal blue induced path of even length with an endvertex which is a neighbour of the endvertex of $S_{R_i}$. Only one of the endvertices of $P'(i)$ is a neighbour of the endvertex of $S_{R_i}$,  and for $j \in I_0,j\neq i$ paths 
$P'(i)$ and $P'(j)$ have different length. It follows that $S_{R_i}$ and $P'(i)$ are fixed.
Endvertex $v_{k(i),i}$ of the path $P'(i)$ is a neighbour of only one of the endvertices of $P(i)$, which is not the endvertex of $R_i$. 
The paths from the family $\{P_j\colon  j \in I_0 \}$ are the only maximal blue induced paths whose endvertex may be a neighbour of $v_{k(i),i}$.
Path $P(i)$ is the only element of $\{P_j\colon j \in I_0 \}$ for which $|P(i)|+1=|P'(i)|$. Moreover, only one of the endvertices of $P(i)$ is a neighbour of $v_{k(i),i}$. It follows than $R_i$ is fixed by $c$. Therefore, we obtained that $\mathcal{R}_0$ is  fixed by $c$.

It remains to show that the rays from the family $\mathcal{R}_\infty$ are fixed by $c$. If $m$ is the favourite vertex of some ray in $\mathcal{R}_\infty$, then it is distinguished from each vertex which is not the favourite vertex of any ray in $\mathcal{R}_\infty$. 
This follows from the fact that $m$ is the unique vertex of some ray in $\mathcal{R}_\infty$ such that $m$ is not the endvertex of this ray and one of the following conditions is satisfied:
\begin{enumerate}[labelindent=\parindent,leftmargin=*, align=left,label=(M\arabic*)]
\item $m$ has an incident blue edge $mw_B$ connecting $m$ and some component $B$ of $G-V(F)$ such that the rooted graph $(B,w_B)$ has  colouring $c(B,w_B,1)$,
\item $B$ is isomorphic to $K_1$ and the edge $mw_B$ is blue (if $|\Orb_\Gamma(B)|\geq n-2$) or yellow (if $|\Orb_\Gamma(B)|\leq n-3$).
\end{enumerate}
 Recall that if $j,k \in I_\infty, j \neq k$, then $m_k \neq m_j$. It follows that the favourite vertices are distinguished from each other. Therefore, they are fixed in $c$. In consequence, all elements of $\mathcal{R}_\infty$ are fixed by $c$. Hence, we obtained a distinguishing colouring of $G$ using at most $n-1$ colours.
\end{proof}

\bibliographystyle{abbrv}
\bibliography{sources}

\end{document}